\documentclass[reqno]{amsart}
\usepackage{color}
\usepackage{amssymb,amsmath,amsthm,amstext,amsfonts}
\usepackage[dvips]{graphicx}
\usepackage{psfrag}
\usepackage{url}

\makeatletter \@addtoreset{equation}{section} \makeatother

\renewcommand\thetable{\thesection.\@arabic\c@table}

\theoremstyle{plain}

\newtheorem{theorem}{Theorem}[section]

\newtheorem{lemma}{Lemma}[section]

\newtheorem{remark}{Remark}[section]

\newcommand{\Z}{\mathbb{Z}}
\newcommand{\N}{\mathbb{N}}
\newcommand{\R}{\mathbb{R}}

\newcommand{\dist}{\operatorname{dist}}

\newcommand{\cB}{\mathcal{B}}

\newcommand{\cA}{\mathcal{A}}

\begin{document}

	\title{H\"{o}lder continuity of Oseledets subspaces for linear cocycles on Banach spaces}

	\author{Chiyi Luo}
	\address{School of Mathematical Sciences, Soochow University\\
		Suzhou 215006, Jiangsu, P.R. China}
	\address{Center for Dynamical Systems and Differential Equations, Soochow University\\
		Suzhou 215006, Jiangsu, P.R. China}
	\email{luochiyi98@gmail.com}
	

	\author{Yun Zhao}
	\address{School of Mathematical Sciences, Soochow University\\
		Suzhou 215006, Jiangsu, P.R. China}
	\address{Center for Dynamical Systems and Differential Equations, Soochow University\\
		Suzhou 215006, Jiangsu, P.R. China}
	\email{zhaoyun@suda.edu.cn}
	
	\thanks{This work is partially supported by NSFC (11790274, 11871361).}
	
	
	\begin{abstract}
		Let $f:X\to X$ be an invertible Lipschitz transformation on a compact metric space $X$. Given a H\"{o}lder continuous invertible operator cocycles on a Banach space and an $f$-invariant ergodic measure, this paper establishes the H\"{o}lder continuity of Oseledets subspaces over a compact set of arbitrarily large measure.  This extends a result in \cite{Simion16} for invertible operator cocycles on a Banach space. Finally, this paper proves the H\"{o}lder continuity in the non-invertible case. 
	\end{abstract}

	

	\maketitle

	\section{Introduction }\label{intro}
	The  multiplicative ergodic theorem plays a fundamental role in the modern theory of dynamical systems, which says that the Lyapunov exponents exist almost everywhere with respect to every given invariant measure, see \cite{Pesin07} for more details of Lyapunov exponents. Before Oseledets' work on multiplicative ergodic theorem \cite{Oseledec68}, there are previous work on random multiplication of matrices by Furstenberg and Kesten \cite{FK60}.
	
	Roughly speaking, the multiplicative ergodic theorem generalizes the notion of eigenvalues and eigenvectors for a single matrix $A\in \R^{d\times d}$ to a product of matrices $ A(f^{n-1}(x)\cdots A(f(x))A(x)$, where $A:X \rightarrow GL(d) $ is an invertible matrix valued function on a probability space $(X,\mu)$ and $f: X \rightarrow X$ is a measure-preserving map with respect to $\mu$. Under some mild reasonable assumptions, there exists a finite set of numbers and subspaces of $\mathbb{R}^d$  which are called Lyapunov exponents and Oseledets subspaces respectively, form a decomposition or a filtration of $\mathbb{R}^d$(depending on whether $f$ is invertible or not). These exponents define the corresponding subspaces of vectors having the same exponential growth rate  under the action of the cocycle generated by $A$.
	
	For infinite dimensional dynamical
	systems, 
	Ruelle \cite{Ruelle82} first proved the multiplicative ergodic theorem for
	compact linear operators on a separable Hilbert space by using the operator theory of Hilbert
	spaces, the main difficulty in this case are that the non-compactness of the phase space and the non-invertibility of the transformation.  Man\'{e} \cite{M83} overcame the lack of inner product structure and extended the multiplicative ergodic theorem to compact operators on
	Banach spaces. Later, Thieullen \cite{Thieullen87} obtained the multiplicative ergodic theorem  for bounded linear operators on Banach spaces.  Lian and Lu \cite{Lian10} established the  multiplicative ergodic theorem for strong measurable operator cocycles on separable Banach spaces, see \cite{S90} and \cite{SF91}  for other versions of the  multiplicative ergodic  theorem for random dynamics in infinite dimensional spaces as mentioned in \cite{Lian10}.
	Froyland \textit{et al} \cite{Froyland10} established the multiplicative ergodic theorem in finite dimensional spaces for an arbitrary matrix cocycle over an invertible ergodic measure-preserving
	transformation of a probability space. Later, the authors in \cite{Froyland13} extended it to a continuous semi-invertible operator on Banach spaces, and Tokman \textit{et al} \cite{Quas12} set up  the  multiplicative ergodic theorem  for a strong measurable semi-invertible operator on Banach spaces.
	
	The dependence of the exponents and the corresponding Oseledets subspaces on the orbit is usually measurable. The stronger regularity of the dependence has been investigated. Brin \cite{Brin01}  studied the special case of a partially hyperbolic  $C^{1+\varepsilon}$ diffeomorphism on a compact manifold. To be more precisely, for every ergodic measure he proved that the subspace given by the direct sum of the Oseledets subspaces corresponding to strictly negative Lyapunov exponents depend H\"{o}lder continuously on the chosen orbit (see also \cite[Chapter 5]{Pesin07}). If $f: X \rightarrow X$ is a Lipschitz map on a compact metric space and $A:X \rightarrow GL(d) $ is H\"{o}lder continuous, for every ergodic measure Araujo \textit{et al} \cite{Simion16}  established the H\"{o}lder continuity of the Oseledets subspaces over a compact set of arbitrarily large measure. Recently, Dragi\v{c}evi\'{c} \textit{et al} \cite{Froyland18} proved the same result in the setting of possibly non-invertible cocycles, which,  in addition, may take values in the space of compact operators on a Hilbert space.

	Let $f: X \rightarrow X$ be a Lipschitz map on a compact metric space $X$, for an invertible operator cocycles $\cA(x,n)$ on Banach spaces and an ergodic measure we show that the Oseledets subspaces depend H\"{o}lder continuously on a compact set of measure arbitrarily close to 1. In Section \ref{PM}, we recall some basic concepts and the multiplicative ergodic theorem for cocycles  on Banach spaces which was proved by Thieullen \cite{Thieullen87}, we also present the statements of the main results here.  In Section \ref{proof}, we give the detailed proofs of the main result. For an ergodic measure, using the Lyapunov norm to constructing a regular set which is compact and the measure arbitrarily close to 1, we estimate the distance of the Oseledets subspaces, which depend H\"{o}lder continuously on points of the regular set. Finally, we prove the H\"{o}lder continuity of Oseledets subspaces in the non-invertible case in Section \ref{non-inert}.

	\section{Preliminaries and statement of the main result}\label{PM}
	Throughout this paper, unless otherwise specified,  let $f$ be a homeomorphism on a compact metric space $(X,d)$, and let $\mu$ be an $f$-invariant ergodic Borel probability measure, and $\cB$ denotes an infinite dimensional Banach space with norm $ || \cdot ||$.
	
	\subsection{Linear cocycles} \label{LC}
	Let $GL(\cB)$ denote the group of invertible bounded linear operators on $ \cB $, the metric $\rho$ on $GL(\cB)$ is defined as follows
	\begin{eqnarray}\label{comp-metric}
\rho(A,B)=||A-B||+||A^{-1}-B^{-1}||,
\end{eqnarray}
	then $(GL(\cB),\rho) $ is a complete metric space.
	
	Let $A:X \rightarrow GL(\cB)$ be a continuous operator valued function, we have that  $\|A\|:=\sup_{x\in X}||A(x)||<\infty$ since $X$ is compact. Consequently, we have that $\log||A||\in L_1(\mu)$. A map $\cA:X\times \Z \rightarrow GL(\cB)$ is called a \textit{linear cocycle} over $ f $ generated by $A$, if
	\begin{equation*}
		\cA(x,n)=\left\{
		\begin{aligned}
			&A(f^{n-1}(x))\circ \cdots \circ A(f(x))\circ A(x), &n > 0\\
			&Id,   &n=0 \\
			&A(f^{-n}(x))^{-1} \circ \cdots \circ A(f^{-1}(x))^{-1},  &n<0
		\end{aligned}
		\right.
	\end{equation*}
	for every $x\in X$. Clearly, $\cA(x,n+k)=\cA(f^{k}(x),n)\circ \cA(x,k)$. Let
	$$ \lambda(x):=\lim_{n\rightarrow \infty} \frac{1}{n} \log ||\cA(x,n)||
	$$
	whenever the limit exists.
	
	Define the index of compactness (or Kuratowski measure of non-compactness) of a continuous linear operator $T:\cB \rightarrow \cB $ as the number
	$$
	||T||_{\alpha}:=\inf \{k>0: T(S_{\cB}) \ \text{can be covered by finitely many balls of radius } k\},
	$$
	where $S_{\cB}$ is the unit ball of $\cB$. It is easy to show that $||T||_{\alpha}\leq ||T|| $ and $\|T\circ S\|_{\alpha}\leq \|T\|_{\alpha}\|S\|_{\alpha}$ for every $T,S\in GL(\cB)$.
	Given a  linear cocycle $\cA:X\times \Z \rightarrow GL(\cB)$ over $ f $ generated by $A:X \rightarrow GL(\cB)$,
	the index of compactness at point $ x $ is defined as
	$$
	\alpha(x):=\lim_{n\rightarrow \infty} \frac{1}{n} \log ||\cA(x,n)||_{\alpha}
	$$
	whenever the limit exists.
	
	By the subadditive ergodic theorem, $\alpha(x)$ and $\lambda(x)$ is well defined for $\mu$-almost every $x$, and the function $x \mapsto \alpha(x)$ and $ x \mapsto \lambda(x)$ are measurable and $f$-invariant. Since $\mu$ is an $f$-invariant ergodic measure, $\alpha(x)$ and $\lambda(x)$ are constants for $\mu$-almost every $x$, denote the constants by $\alpha(A,\mu)$ and $\lambda(A,\mu)$ respectively. Note that $ \alpha(A,\mu) \leq \lambda(A,\mu)$, and we call $A$ is \textit{quasi-compact} if  $ \alpha(A,\mu)< \lambda(A,\mu)$.
	
	The following Oseledets multiplicative ergodic theorem of continuous cocycles on Banach spaces was proved by P. Thieullen \cite{Thieullen87}.
	\begin{theorem}[Multiplicative ergodic theorem]\label{EMT}
		Let $f$ be a homeomorphism on a compact metric space $(X,d)$, and let $\mu$ be an $f$-invariant ergodic Borel probability measure. Given a linear cocycle $\cA$ over $ f $ generated by the quasi-compact and continuous operator valued function $A:X \rightarrow GL(\cB)$, where $(\cB,||\cdot||)$ is a Banach space. Then there exists a $f$-invariant subset $X_0\subset X$ of  full $\mu$-measure such that  one of the following cases hold:
		
		{\bf Case (1)}: There exist  finite  numbers
		$$\lambda_{1}>\lambda_{2}>\cdots>\lambda_{k}$$
		with $\lambda_{1}=\lambda(A,\mu)$ and $\lambda_{k}>\alpha(A,\mu)$, and  for every $x\in X_0$ there is a splitting
		$$\cB=E_{1}(x) \oplus E_{2}(x) \oplus \cdots \oplus E_{k}(x) \oplus F(x) $$
		such that
		\begin{enumerate}
			\item[(a)] for each $i=\{1,\cdots,k\}$, $ \dim E_{i}(x)=m_i $ is finite and constant. Moreover, $A(x)E_{i}(x)=E_{i}(f(x))$, and for any $v\in E_{i}(x)\setminus\{0\}$ we have
			$$ \lambda_{i}=
			\lim_{n\rightarrow \infty} \frac{1}{n}\log ||\cA(x,n)v||=
			-\lim_{n\rightarrow \infty} \frac{1}{n}\log ||\cA(x,-n)v||;
			$$
			\item[(b)] the distribution $ F(x) $ is closed and finite-codimensional, satisfies $A(x)F(x)=F(f(x))$ and
			$$ \alpha(A,\mu)>\limsup_{n\rightarrow \infty} \frac{1}{n}\log ||\cA(x,n)|_{F(x)}||; $$
			\item[(c)] the maps $x\mapsto E_{i}(x), x\mapsto F(x)$ are measurable;
			\item[(d)] writing $\pi_{i}(x)$ for the projection of $\cB $ onto $E_{i}(x)\,(i=1,2,\cdots,k)$ via the splitting at $ x $, we have
			$$\lim_{n\rightarrow \pm\infty} \frac{1}{n}\log ||\pi_{i}(f^{n}(x))||=0.$$
		\end{enumerate}
		
		{\bf Case (2)}: There exist  infinite numbers
		$$\lambda_{1}>\lambda_{2}>\cdots>\alpha(A,\mu)$$
		with $\lambda_{1}=\lambda(A,\mu)$ such that the following properties hold: for each $x\in X_0$ and each
		positive integer $k\in \N$ there is a splitting
		$$\cB=E_{1}(x) \oplus E_{2}(x) \oplus \cdots \oplus E_{k}(x) \oplus F_{k+1}(x) $$
		such that
		\begin{enumerate}
			\item[(a)] for each $i=\{1,\cdots,k\}$, $ \dim E_{i}(x)=m_i $ is finite and constant for $\mu$-almost every $x$. Moreover, $A(x)E_{i}(x)=E_{i}(f(x))$ and for every $v\in E_{i}(x)\setminus\{0\}$ we have
			$$ \lambda_{i}=
			\lim_{n\rightarrow \infty} \frac{1}{n}\log ||\cA(x,n)v||=
			-\lim_{n\rightarrow \infty} \frac{1}{n}\log ||\cA(x,-n)v||;
			$$
			\item[(b)] the distribution $F_{k+1}(x)$ is closed, finite-codimensional and $A(x)F_{k+1}(x)=F_{k+1}(f(x))$ and
			$$ \lambda_{k+1}>\limsup_{n\rightarrow \infty} \frac{1}{n}\log ||\cA(x,n)|_{F_{k+1}(x)}||; $$
			\item[(c)] the maps $x\mapsto E_{i}(x), x\mapsto F_{k+1}(x)$ are measurable;
			\item[(d)] writing $\pi_{i}(x)$ for the projection of $\cB $ onto $ E_{i}(x) $ via the splitting at $ x $, we have
			$$\lim_{n\rightarrow \pm\infty} \frac{1}{n}\log ||\pi_{i}(f^{n}(x))||=0.$$
		\end{enumerate}
	\end{theorem}
	The number $\lambda_{i}$ in the above theorem is called the  \textit{$i$-th Lyapunov exponents} of the cocycle $\cA$ with respect to $\mu$ and  $m_i$ are called multiplicities of
	$\lambda_{i}$ for every $i$. Moreover, the splitting is called the \textit{Oseledets splitting} and $E_{i}(x)$ is called the \textit{Oseledets subspaces.}
	
	\subsection{Gaps and distance between closed linear subspaces}
	We gather in this subsection some  facts that are relevant to Banach spaces. The definitions of gap and distance are taken from Kato \cite{Kato95} (see also \cite{Lian10} and \cite{Young17}).
	
	Let $ E $ and $ F $ be two non-trivial closed linear subspaces of the Banach space $\cB$, and let $S_{E}$ denote the unit ball of $E$. Put
	\begin{eqnarray*}
		\delta(E,F)=\sup_{v\in S_{E}} \dist(v,F)=\inf\{a: \dist(v,F)\leq a||v||\  \text{for any} \ v \in E\setminus\{0\}\}
\end{eqnarray*}
where $\dist(v,F)=\inf_{u\in F}||v-u||$,
and define the gap between $E$ and $F$ as follows
\begin{eqnarray*}
		\hat{\delta}(E,F)=\max\{\delta(E,F),\delta(F,E)\}.
	\end{eqnarray*}
	If $\cB$ is a Hilbert space, then $\hat{\delta}$ is a metric and coincides with the operator norm of the difference between orthogonal projections. However, in the case that $\cB$ is a  Banach space, $\hat{\delta}$ is not a metric since it does not satisfies the triangle inequality in general \cite{Kato95}.
	
	The topology on the set of closed linear subspaces of the Banach space $\cB$ is the  metric topology defined by the Hausdorff distance $ \hat{d} $ between unit spheres
	\begin{align*}
		\hat{d}(E,F)=\max\{\dist(E,F),\dist(F,E)\}
	\end{align*}
    where $\dist(E,F)=\sup_{v\in S_{E}} \dist(v,S_{F})$.
	The gap $\hat{\delta}$ and the distance $ \hat{d} $ are related by the following inequality (see \cite{Kato95}):
	\begin{equation}\label{eq:gap}
		\hat{\delta}(E,F) \leq \hat{d}(E,F) \leq 2\hat{\delta}(E,F).
	\end{equation}
	Hence, in the following we will work with $ \hat{\delta}(E,F) $ since it is more convenient.
	
	The following lemma gives conditions under which complementation persists, see \cite[Lemma 3.3]{Young17} for detailed proofs.
	\begin{lemma}\label{lem Yong}
Assume that $E$ is a finite dimensional subspace of $\cB$,  $F$ is a closed subspace of $\cB$ and $\cB=E\oplus F$. Let $\pi_{E//F}$ be the projection operator $E\oplus F\rightarrow E$. If ${E}'$ is a finite dimensional subspace of $\cB$ such that $\hat{d}(E,{E}')\leq ||\pi_{E//F} ||^{-1}$, then $\cB={E}' \oplus F$.
	\end{lemma}
	
	\subsection{Statement of the main result}
For a  cocycle generated by an invertible bounded linear operator on the Banach space $\cB$ and an ergodic measure, this paper proves that the corresponding  Oseledets subspaces varies continuously on a compact set of arbitrarily large measure. More precisely,   if $ f: X\rightarrow X $ is bi-Lipschitz on a compact metric space and $ A:X \rightarrow GL(\cB) $ is H\"{o}lder continuous, then the Oseledets subspaces $E_{i}(x), F_{i}(x)$ and $F(x)$ are  H\"{o}lder continuous on a compact set of  measure close to 1. This extends the main results in \cite{Simion16} for linear cocycles on a Banach space. Dragi\v{c}evi\'{c} \textit{et al} \cite{Froyland18} also proved the same result for  possibly non-invertible cocycles on $\R^{d}$ as well as compact operator cocycles on Hilbert spaces.

	\begin{theorem}\label{main}
		Let $f:X\rightarrow X$ be a bi-Lipschitz homeomorphism on a compact metric space $X$ and $\mu$ an ergodic Borel probability measure on $X$, and let $\cA$ be a linear cocycle over $ f $ generated by the $\nu$-H\"{o}lder continuous function $A:X \rightarrow GL(\cB)$.
Let $\lambda_{1}>\lambda_{2}>\cdots>\lambda_{k}>\lambda_{k+1}$ denote the distinct $k+1$ Lyapunov exponents, corresponding to the splitting $\cB=E_{1}(x) \oplus \cdots \oplus E_{k}(x) \oplus F_{k+1}(x) $ defined for $ \mu $-almost every $x\in X$. Then, for every $\gamma>0$, there exist a compact subset $\Lambda_{\gamma}$ of $X$ with $ \mu(\Lambda_{\gamma})>1-\gamma$, and constants $C=C(\Lambda_{\gamma})>0$, $\omega_i=\omega_i(\lambda_{1},\cdots,\lambda_{k+1})<1, i=1,\cdots,k+1$ and $\delta=\delta(\gamma,\Lambda_{\gamma},\lambda_{1},\cdots,\lambda_{k+1})$ such that for all $x,y\in \Lambda_{\gamma}$ with $d(x,y)<\delta$, we have that for $i=1,\cdots,k$
		$$\hat{d}(E_i(x),E_i(y))\leq C d(x,y)^{\nu\omega_i}\ \text{and} \ \  \hat{d}(F_{k+1}(x),F_{k+1}(y))\leq C d(x,y)^{\nu\omega_{k+1}}.$$
		\end{theorem}

\begin{remark} In the above theorem,
 $k$ is the number of all finite dimensional subspaces and $\lambda_{k+1}:=\alpha(A,\mu), F_{k+1}(x):=F(x)$ in {\bf Case (1)}, and $k$ is any given positive integer  in {\bf Case (2)}.
\end{remark}
	
	\section{Proofs}\label{proof}
	We only prove Theorem \ref{main} for the {\bf Case (1)}, that is, there are finite  Lyapunov exponents:
	$$\lambda_{1}>\lambda_{2}>\cdots>\lambda_{k}>\lambda_{k+1}:=\alpha(A,\mu),$$
	correspoding to the splitting $\cB=E_{1}(x) \oplus \cdots \oplus E_{k}(x) \oplus F(x)$ defined on an $f$-invariant subset $X_0$ of full $\mu$-measure. The {\bf Case (2)} can be proven in a similar fashion.

	\subsection{Construction of the regular set}
	We first recall  the definition of the Lyapunov norm. Given a sufficiently small number $\varepsilon>0$, for each $x\in X_0$ and every $u=u_1+\cdots+u_k+u_{k+1} \in \cB, u_i \in E_{i}(x)\,(i=1,\cdots,k), u_{k+1}\in F(x)$, the Lyapunov norm is defined as
	$$ ||u||_{x,\varepsilon}:=||u||_x=||u_1||_x+\cdots+||u_k||_x+||u_{k+1}||_x$$
	where
	$||u_i||_x=\sum_{n=-\infty}^{\infty} e^{-n\lambda_{i}-|n|\varepsilon} ||\cA(x,n)u_i||$ for $i=1,\cdots,k$ and
	$$||u_{k+1}||_x=\sum_{n=0}^{\infty} e^{-n(\lambda_{k+1}+\varepsilon)} ||\cA(x,n)u_{k+1}||.$$
	For every $n>0$ and every $u_i \in E_{i}(x)\,(i=1,\cdots,k)$, one can easily show that
	\begin{equation}\label{eq:u1}
		e^{n(\lambda_{i}-\varepsilon)} ||u_i||_x \leq ||\cA(x,n)u_i||_{f^n(x)} \leq e^{n(\lambda_{i}+\varepsilon)} ||u_i||_x,
	\end{equation}
	and for every $u_{k+1}\in F(x)$
	\begin{equation}\label{eq:u2}
		||\cA(x,n)u_{k+1}||_{f^n(x)} \leq e^{n(\lambda_{k+1}+\varepsilon)} ||u_{k+1}||_x.
	\end{equation}
	
	The following lemma provides some fundamental properties of the above Lyapunov norm, see \cite[Theorem 7.2.3]{Doan09} for detailed proofs.
	\begin{lemma}\label{norm}
		Given a small number $\varepsilon>0$, there exists a measurable function $D_{\varepsilon}: X_0\rightarrow [1,\infty)$ such that for every $ x \in X_0 $
		\begin{equation}
			||\cdot|| \leq || \cdot ||_{x} \leq D_{\varepsilon}(x) ||\cdot||
		\end{equation}
		and for each $n\in \Z$
		$$D_{\varepsilon}(f^n(x)) \leq e^{|n|\varepsilon} D_{\varepsilon}(x).$$
	\end{lemma}
	
	In the setting of the space $\R^{d}$ or the Hilbert space, we have that the angels between two  Oseledets subspaces decay sub-exponentially along orbits of $ x $. However, in the case of Banach spaces, there is not such statement since the lack of reasonable definition of the ``angels" between two closed linear subspaces of Banach space. The fact that the norm of each projection operator $\pi_{i}(x),i=1,\cdots,k$ is temperate (see (d) of Theorem \ref{EMT}) helps us to overcome this difficulty. For the multiplicative ergodic theorem of semi-invertible operator on Banach spaces (see \cite{Froyland13}), it is also valid by the following lemma proved by Dragi\v{c}evi\'{c} \textit{et al} \cite[Lemma 1]{Froyland18}.
	
	\begin{lemma}\label{L3}
		Let $\Lambda$ be an $f$-invariant set and let $E(x)$ and $F(x), x\in \Lambda$ be
		two families of closed subspaces of $\cB$. Assume that  there exist numbers $\chi_{2} < \chi_{1}, \varepsilon>0$ with $\chi_{2}+3\varepsilon \leq \chi_{1}-2\varepsilon$ and measurable functions $C,\tilde{C}:\Lambda \rightarrow [1,\infty)$ such that
		\begin{enumerate}
			\item[(1)] $A(x)E(x)\subset E(f(x)), A(x)F(x)\subset F(f(x))$ and $E(x)\cap F(x)=\{0\}$ for every $x\in \Lambda$;
			\item[(2)] for every $x\in \Lambda, v\in E(x)\oplus F(x)$ and $n\geq 0$,
			$$||\cA(x,n)v||\leq  \tilde{C}(x)e^{(\chi_{1}+\varepsilon)n}||v||;$$
			\item[(3)] for every $x\in \Lambda, v\in F(x)$ and $n\geq 0$,
			$$||\cA(x,n)v||\geq \frac{1}{\tilde{C}(x)}e^{(\chi_{1}-\varepsilon)n}||v||;$$
			\item[(4)] for every $x\in \Lambda, v\in E(x)$ and $n\geq 0$,
			$$||\cA(x,n)v||\leq  C(x)e^{(\chi_{2}+\varepsilon)n}||v||; \ and$$
			\item[(5)] for every $x\in \Lambda$ and $m\in \Z$,
			$$C(f^n(x)) \leq e^{|n|\varepsilon} C(x) \ and \
			\tilde{C}(f^n(x)) \leq e^{|n|\varepsilon} \tilde{C}(x).$$
		\end{enumerate}
		Then, there exists a measurable function $K:\Lambda \rightarrow [1,\infty)$ satisfies
		$$K(f^n(x)) \leq e^{|n|\varepsilon} K(x) \ \text{for each} \ n\in \Z \ and \ x\in \Lambda$$
		and such that
		$$||v_1||\leq K(x)||v_1+v_2|| \ and \ ||v_2||\leq K(x)||v_1+v_2|| $$
		for $v_1\in E(x)$ and $v_2\in F(x)$.
	\end{lemma}
\begin{remark}	Dragi\v{c}evi\'{c} \textit{et al} \cite{Froyland18} proved the above lemma for the Euclid space $\R^d$, it is also valid for the case of Banach space $\cB$, with a minor modification of the proof. Use the lemma above with some inductions, one can easily to show that (d) of Theorem \ref{EMT} is valid for semi-invertible cocycle on Banach space.
\end{remark}
	
	\begin{theorem}\label{T2}
		Let $\cA$ be a cocycle over $ f $ with Lyapunov exponents as in Case (1), for each $i \in \{1,\cdots,k\}$. Let
		$$ E_i^{+}(x)=\bigoplus_{j=1}^{i}E_{j}(x) \ and \    E_i^{-}(x)=(\bigoplus_{j=i+1}^{k}E_{j}(x)) \bigoplus F(x).$$
		Then there exists a full $\mu$-measure, $f$-invariant subset $\Lambda$ of $X$ such that, for each $\varepsilon>0$ small enough with $\varepsilon<\min_{i=1,\cdots,k}\{(\lambda_{i}-\lambda_{i+1})/100\}$, there are measurable functions $C,K:\Lambda \rightarrow [1,\infty)$ with
		$$C(f^n(x)) \leq e^{|n|\varepsilon} C(x) \ and \ K(f^n(x)) \leq e^{|n|\varepsilon} K(x).$$
		so that for every $x\in \Lambda$:
		\begin{enumerate}
			\item[(1)] for each $u\in E_i^{-}(x),v\in E_i^{+}(x)$ and $n\geq 0$,
			$$||\cA(x,n)u||\leq  C(x)e^{(\lambda_{i+1}+\varepsilon)n}||u|| \ and \
			||\cA(x,n)v||\geq \frac{1}{C(x)}e^{(\lambda_{i}-\varepsilon)n}||v||;$$
			\item[(2)] for each $u\in E_i^{-}(x)$ and $v\in E_i^{+}(x)$,
			$$\max\{||u||,||v||\}\leq K(x) ||u+v||.$$
		\end{enumerate}
	\end{theorem}
	\begin{proof}
    Following the proof of Proposition 3.2 in \cite{Lucas17}, we give the detailed proofs as follows.
    By \eqref{eq:u1}, \eqref{eq:u2} and Lemma \ref{norm}, there exist a set $\Lambda$ and a measurable function $C:\Lambda \rightarrow [1,\infty)$ with $C(f^{\pm}(x)) \leq e^{\varepsilon} C(x)$ so that the first statement hold. To complete the proof of the theorem, it suffices to prove the second statement.

    Let $\pi_{i}^{+}(x)$ and $\pi_{i}^{-}(x)$ denote the projections of $\cB$ onto $E_{i}^{+}(x)$ and $E_{i}^{-}(x)$ via the splitting $\cB=E_{i}^{+}(x) \oplus E_{i}^{-}(x)$. By (d) of Theorem \ref{EMT}, for every $x\in \Lambda$ we have that
    $$\lim_{n\rightarrow \pm\infty} \frac{1}{n} \log||\pi_{i}^{+}(f^{n}(x))||=0, \ \lim_{n\rightarrow \pm\infty} \frac{1}{n} \log||\pi_{i}^{-}(f^{n}(x))||=0.$$
        Therefore, define the function $K:\Lambda \rightarrow [1,\infty)$ as follows:
    $$K(x)=\sup_{n\in \Z} \dfrac{\max\{||\pi_{i}^{-}(f^{n}(x))||,||\pi_{i}^{+}(f^{n}(x))||\}}{e^{|n|\varepsilon}}.$$
    One can easily to show that $K$ is a well defined function on $\Lambda$. Moreover, one has that $K(x)\geq\max\{||\pi_{i}^{-}(x)||,||\pi_{i}^{+}(x)||\}$ and $K(f^{\pm}(x)) \leq e^{\varepsilon} K(x)$ for every $x\in \Lambda$.
		
    Finally, for each $u\in E_{i}^{-}(x)$ and each $v\in E_{i}^{+}(x)$ one has that
    \begin{align}\label{norm-proj1}
    	||u||=||\pi_{i}^{-}(x)(u+v)||\leq ||\pi_{i}^{-}(x)||\cdot||u+v|| \leq K(x)||u+v||
    \end{align} and
    \begin{align}\label{norm-proj2}
    	||v||=||\pi_{i}^{+}(x)(u+v)||\leq ||\pi_{i}^{+}(x)||\cdot||u+v|| \leq K(x)||u+v||.
    \end{align}
    This completes the proof of the theorem.
	\end{proof}
	
	Fix a sufficiently small $ \varepsilon>0 $. For every $\ell\in \N$, we define the regular set $\Lambda_{\ell}$ by $$\Lambda_{\ell}=\{x\in \Lambda:C(x)\leq \ell \ and \ K(x)\leq \ell\}.$$ One can easily show that each $\Lambda_{\ell}$ is compact, $\Lambda_{\ell} \subset \Lambda_{\ell+1}$ and $\bigcup_{\ell>0}\Lambda_{\ell}=\Lambda$. Thus, for every $\gamma>0$, we may choose a subset $\Lambda_{\ell}$ with $\mu(\Lambda_{\ell})>1-\gamma$.

	\subsection{H\"{o}lder continuity of maps $x\mapsto E_i^{-}(x)$ and $x\mapsto E_i^{+}(x)$}
Fix $ i\in \{1,\cdots,k\}$, and let $E_i^{-}(x) $ and $E_i^{+}(x)$ be as in Theorem \ref{T2}, we now prove that the map $x\mapsto E_i^{-}(x)$ and $x\mapsto E_i^{+}(x)$ are (locally) H\"{o}lder continuous on $\Lambda_{\ell}$.

	The following two lemmas is useful in the proof of the main result. See \cite{Brin01} for the original versions of the finite dimensional case, we also refer the reader to Lemmas 5.3.4 and  5.3.5 in \cite{Pesin07} or  Lemmas 2.1 and  2.2 in \cite{Simion16}.
	\begin{lemma}\label{L4}
		Assume that $A:X \rightarrow GL(\cB)$ is $\nu$-H\"{o}lder continuous with H\"{o}lder constant $a_1$ and
		$f:X\rightarrow X$ is bi-Lipschitz with constant $L\geq 1$, then there exists a constant $a>a_1$ such that
		\begin{equation}\label{eq:L4}
			||\cA(x,n)-\cA(y,n)||\leq a^{|n|} d(x,y)^{\nu}
		\end{equation}
		for every $x,y\in X$ and $n\in \Z$.
	\end{lemma}
	\begin{proof}
		We follow the proof of Lemma 2.2 in \cite{Simion16} and argue by induction on $n$.
		
		For $k=1$, by the H\"{o}lder continuity of the map $A:X \rightarrow GL(\cB)$ and \eqref{comp-metric} we have $$||A(x)-A(y)||\leq a_1 d(x,y)^{\nu}.$$
		Assume that there exists $a>a_1$ so that \eqref{eq:L4} hold for $ k=1,2,\cdots,n$. For $k=n+1$, we have that 		 \begin{align*}
			||&\cA(x,n+1)-\cA(y,n+1)||\\
			&\leq || A(f^n(x))||\cdot ||\cA(x,n)-\cA(y,n)||+||\cA(y,n)||\cdot||A(f^n(x))-A(f^n(y)) ||\\
			&\leq (\sup_{x\in X} ||A(x||) a^{n} d(x,y)^{\nu}+(\sup_{x\in X} ||A(x)||)^n a_1 d(f^n(x),f^n(y))^{\nu}\\
			&\leq [(\sup_{x\in X} ||A(x||)a^{n}+(\sup_{x\in X} ||A(x)||)^n a_1L^{n \nu }]d(x,y)^{\nu}
		\end{align*}
where we use the fact that $f$ is Lipschitz in the last inequality.
		To find the number $a$, all we need is that
		$$a^{n+1}>(\sup_{x\in X} ||A(x||)a^{n}+(\sup_{x\in X} ||A(x)||)^n a_1L^{n \nu },$$
		that is
		$$a\geq a_1(\dfrac{\sup_{x\in X} ||A(x)||\cdot L^{\nu}}{a})^n+\sup_{x\in M} ||A(x)||.$$
		This can be easily  achieved by taking a sufficiently large $a$ such that
		$$a > \max\{a_1,\sup_{x\in M} ||A(x)||\cdot L^{\nu} \}.$$
		 The case for $n<0$ can be proven in a similar fashion, this completes the proof.
	\end{proof}


	\begin{lemma}\label{L5}
		Let $\{A_n\}_{n\geq1},\,\{B_n\}_{n\geq1}$ be two sequences of operators in $GL(\cB)$, such that, for some $0<\alpha_{2}<\alpha_{1}$ and $\ell\geq 1$, there exist closed subspaces $E,{E}',F,{F}'$ and $\cB_0$ of $\cB$ satisfying $\cB_0=E \oplus {E}'=F \oplus {F}'$ such that for some fixed $n$
		\begin{enumerate}
			\item[(i)] $||A_{n}u||\leq \ell \alpha_{2}^n||u||$  and $||A_{n}v||\geq \ell^{-1} \alpha_{1}^n ||v||$ for every $u\in E$, $v\in {E}'$;
			\item[(ii)]$||B_{n}u||\leq \ell \alpha_{2}^n||u||$ and $||B_{n}v||\geq \ell^{-1} \alpha_{1}^n ||v||$ for every $u\in F$, $v\in {F}'$;
			\item[(iii)] $\max\{||v||,||w||\} \leq \ell||u||$ for each $u=v+w, v\in E,w\in{E}'$ or $v\in F,w\in{F}'$.
		\end{enumerate}	
		Then for every $\delta<1, a \geq \alpha_{1}$ satisfying
		$$(\frac{\alpha_2}{a})^{n+1}\le\delta < (\frac{\alpha_2}{a})^{n} \  \text{and} \ ||A_n-B_n ||\leq \delta a^n,$$
		we have that $$ \hat{d}(E,F)\leq (4+2\ell)\ell^2 \frac{\alpha_{1}}{\alpha_{2}} \delta^{\log(\alpha_{2}/\alpha_{1})/\log(\alpha_{2}/a)}.$$
	\end{lemma}
	\begin{proof}
		With a minor modification of the proof of Lemma 2.1 in \cite{Simion16}, we give the proof of the result in the following.
		
		Let us define the cone $Q=\{u\in \cB_0:||A_n u||\leq 2\ell \alpha_{2}^n||u||\}$. For each $v\in F$, one has
		\begin{align*}
			||A_n v||&\leq ||A_n-B_n||\cdot ||v||+||B_n v||\\
			&\leq (\delta a^n+\ell\alpha_{2}^n)||v||\\
&\leq 2\ell\alpha_{2}^n||v||.
		\end{align*}
In consequence,  $v\in Q$ and this implies that $F\subset Q$.
		
		For each $v\in Q$, write $v=v_1+v_2$ where $v_1\in E$ and $v_2\in {E}'$.  Recall that $\ell||v||\geq\max\{||v_1||,||v_2||\}$, one has
		\begin{align*}
			2\ell\alpha_{2}^n||v||&\geq||A_n (v_1+v_2)||\\
			&\geq||A_n v_2||-||A_n v_1||\\
			&\geq \ell^{-1} \alpha_{1}^n ||v_2||-\ell\alpha_{2}^n \cdot \ell||v||.
		\end{align*}
		This together with the fact that $(\alpha_2/a)^{n+1}< \delta < (\alpha_2/a)^{n} $ imply that
		\begin{align*}
			\dist(v,E)\leq ||v_2|| \leq (2+\ell)\ell^2 \Big(\frac{\alpha_{2}}{\alpha_{1}}\Big)^n ||v||
			\leq (2+\ell)\ell^2\frac{\alpha_{1}}{\alpha_{2}}\delta^{\log(\alpha_{2}/\alpha_{1})/\log(\alpha_{2}/a)}||v||.
		\end{align*}
		 Since $F\subset Q$, by the definition of $\delta(\cdot,\cdot)$ one has
		$$
		\delta(F,E)\leq (2+\ell)\ell^2\frac{\alpha_{1}}{\alpha_{2}}\delta^{\log(\alpha_{2}/\alpha_{1})/\log(\alpha_{2}/a)}.$$
		Symmetrically, one can show that  $\delta(E,F)\leq (2+\ell)\ell^2\frac{\alpha_{1}}{\alpha_{2}}\delta^{\log(\alpha_{2}/\alpha_{1})/\log(\alpha_{2}/a)}$. Hence, using
		\eqref{eq:gap} we conclude that
		$$\hat{d}(E,F)\leq (4+2\ell)\ell^2\frac{\alpha_{1}}{\alpha_{2}}\delta^{\log(\alpha_{2}/\alpha_{1})/\log(\alpha_{2}/a)}.$$
	\end{proof}

	Now, we estimate $\hat{d}(E_i^{-}(x),E_i^{-}(y))$ and $\hat{d}(E_i^{+}(x),E_i^{+}(y))$ for each $x,y\in \Lambda_{l}$ with $d(x,y)< 1$, where $E_i^{-}(x),E_i^{+}(x)$ is the same as in Theorem \ref{T2}.
	\begin{lemma}\label{H-}
		For each $x,y\in \Lambda_{\ell}$ with $d(x,y)< 1$, we have
		$$\hat{d}(E_i^{-}(x),E_i^{-}(y))\leq C_{i}^{-} d(x,y)^{\nu_{i}^{-}}$$
		where $C_{i}^{-}$ and $\nu_{i}^{-}$ are two constants.
	\end{lemma}
\begin{proof}
Given $x,y\in \Lambda_{\ell}$ with $d(x,y)< 1$, let $A_n=\cA(x,n), B_n=\cA(y,n), \alpha_{2}=e^{\lambda_{i+1}+\varepsilon},\alpha_{1}=e^{\lambda_{i}-\varepsilon}, E=E_i^{-}(x),F=E_i^{-}(y),{E}'=E_i^{+}(x)$ and ${F}'=E_i^{+}(y)$. By Theorem \ref{T2} and the definition of  $\Lambda_{\ell}$, the conditions (i), (ii), (iii) of Lemma \ref{L5} hold for every $n\in \N$. Set $\delta=d(x,y)^{\nu}<1$. By Lemma \ref{L4}, there exists a sufficiently large constant $a$ such that $1>e^{\lambda_{i+1}+\varepsilon}/a$ and
$$||A_n-B_n||\leq a^n d(x,y)^{\nu}=a^n \delta$$
for each $n\in \N$.
Moreover, there exists ${n}'={n}'(\delta,a,i)\in \N$ such that
$$(\frac {e^{\lambda_{i+1}+\varepsilon}} {a})^{{n}'+1}\le\delta < (\frac {e^{\lambda_{i+1}+\varepsilon}}{a})^{{n}'}.$$
It follows from Lemma \ref{L5} that
\begin{equation} \label{eqa:EStable}
\begin{aligned}
\hat{d}(E_i^{-}(x),E_i^{-}(y))&\leq (4+2\ell)\ell^2 \frac{\alpha_{1}}{\alpha_{2}}\delta^{\log(\alpha_{2}/\alpha_{1})/\log(\alpha_{2}/a)}\\
&= (4+2\ell)\ell^2 e^{\lambda_{i}-\lambda_{i+1}-2\varepsilon} d(x,y)^{\frac{\nu(\lambda_{i+1}-\lambda_{i}+2\varepsilon)}{\lambda_{i+1}+\varepsilon-\log a}}\\
&= C_{i}^{-} d(x,y)^{\nu_{i}^{-}}
\end{aligned}
\end{equation}
where $C_{i}^{-}=(4+2\ell)\ell^2 e^{\lambda_{i}-\lambda_{i+1}-2\varepsilon}$ and $\nu_{i}^{-}=\nu(\lambda_{i}-\lambda_{i+1}-2\varepsilon)/(\log a-\lambda_{i+1}-\varepsilon )<\nu<1.$
\end{proof}
\begin{remark}
	Notice the hypotheses that 
$A(x)$ is invertible is not used in the estimate of $\hat{d}(E_i^{-}(x),E_i^{-}(y))$, and Theorem \ref{T2} is also valid for  semi-invertible cocycles (see \cite[Proposition 3.2]{Lucas17}). Thus, the conclusion of the above lemma is true for semi-invertible operators on a Banach space.
\end{remark}
	
\begin{lemma}\label{H+}
	For each $x,y\in \Lambda_{\ell}$ with $d(x,y)< 1$, we have
	$$\hat{d}(E_i^{+}(x),E_i^{+}(y))\leq C_{i}^{+} d(x,y)^{\nu_{i}^{+}}$$
	where $C_{i}^{+}$ and $\nu_{i}^{+}$ are two constants.
\end{lemma}
\begin{proof}
In order to estimate $\hat{d}(E_i^{+}(x),E_i^{+}(y))$ for every $x,y\in \Lambda_{\ell}$ with $d(x,y)< 1$, set $A_n=\cA(x,-n),B_n=\cA(y,-n)$ for every $n\in \mathbb{N}$, $\alpha_{2}=e^{-\lambda_{i}+2\varepsilon},\alpha_{1}=e^{-\lambda_{i+1}-2\varepsilon}, E=E_i^{+}(x),F=E_i^{+}(y),{E}'=E_i^{-}(x)$ and ${F}'=E_i^{-}(y)$. Take $u\in E'$, since  $\cA(x,-n)=(\cA(f^{-n}(x),n))^{-1}$, it follows from Theorem \ref{T2} that
$$\begin{aligned}
||u||&=||\mathcal{A}(f^{-n}(x),n)(\mathcal{A}(x,-n) u)||\\
&\leq C(f^{-n}(x))e^{\lambda_{i+1}+\varepsilon}||\mathcal{A}(x,-n) u||\\
&\le e^{n\varepsilon}\ell e^{\lambda_{i+1}+\varepsilon}||\mathcal{A}(x,-n) u||
\end{aligned}
$$
 where we use the facts that $C(f^{-n}(x)) \leq e^{n\varepsilon}C(x)$ and $C(x)\leq \ell$. Consequently,  one has
$$||A_n u||\geq  \ell^{-1} e^{n(-\lambda_{i+1}-2\varepsilon)} ||u|| = \ell^{-1} \alpha_{1}^n||u||.$$
Similarly, one can show that $||A_n v|| \leq  \ell \alpha_{2}^n||v||$ for every $v\in {E}$.
Thus, the condition (i) of Lemma \ref{L5} holds. Replace $x$ with $y$, one can show that the condition (ii) of Lemma \ref{L5} holds in similar fashion, and the condition (iii) follows from (2) of Theorem \ref{T2}. Let $\delta=d(x,y)^{\nu}<1$ and let the constant $a$ be as  Lemma \ref{L5}. It follows from Lemma \ref{L5} that
\begin{equation} \label{eqa:EUnstable}
\hat{d}(E_i^{+}(x),E_i^{+}(y))\leq C_{i}^{+} d(x,y)^{\nu_{i}^{+}}.
\end{equation}
where $C_{i}^{+}=(4+2\ell)\ell^2 e^{\lambda_{i}-\lambda_{i+1}-4\varepsilon}$ and $\nu_{i}^{+}=\nu(\lambda_{i}-\lambda_{i+1}-4\varepsilon)/(\log a+\lambda_{i}-2\varepsilon )<\nu<1.$
\end{proof}
	
	\subsection{H\"{o}lder continuity of the map $x\mapsto E_i(x)$}

	In this section, we will give the proof of  the main result of this paper.
	
	For every $\gamma>0$, fix $\Lambda_\ell$ so that $\mu(\Lambda_\ell)>1-\gamma$.
Take a point $x\in \Lambda_\ell$, then  $\cB=E_i^{-}(x) \oplus E_i^{+}(x)$, here $E_i^{-}(x),E_i^{+}(x)$ are the same as in Theorem \ref{T2}. By \eqref{eqa:EUnstable}, choose  a small number $\delta_1\in(0,1)$ such that
	$$\hat{d}(E_i^{+}(x),E_i^{+}(y))\leq 1/\ell$$
	for each $x,y\in \Lambda_{\ell}$ with $d(x,y)<\delta_1$.
	By the definition of $\Lambda_\ell$ and \eqref{norm-proj2},  the norm of the projection operator $\pi_{i}^{+}(x): E_i^{+}(x) \oplus E_i^{-}(x) \rightarrow E_i^{+}(x)$ is no larger than $\ell$, thus
	$$\hat{d}(E_i^{+}(x),E_i^{+}(y))\leq 1/\ell < ||\pi_{i}^{+}(x)||^{-1}.$$
	By Lemma \ref{lem Yong}, this yields that $\cB=E_i^{+}(y) \oplus E_i^{-}(x)$. So, there exists a
	linear operator $L_{x,y}:E_i^{+}(x)\rightarrow E_i^{-}(x)$ such that the graph of $L_{x,y}$
	is equivalent to the subspace $E_i^{+}(y)$, that is,
	$$E_i^{+}(y)=\{u+L_{x,y}(u): u\in E_i^{+}(x)\}.$$
	\begin{lemma}\label{graph}
		For each $x,y\in \Lambda_{\ell}$ with $d(x,y)<\delta_1$, we have
		\begin{equation}\label{bound-subsp}
			\dfrac{||L_{x,y}||}{\ell(1+||L_{x,y}||)}\leq \hat{d}(E_i^{+}(x),E_i^{+}(y)) \leq 2\ell||L_{x,y}||.
		\end{equation}
	\end{lemma}
	\begin{proof}
		For simplicity of presentation,  denote $L_{x,y}$ by $L$. First, for each $u\in E_i^{+}(x)$, the fact that $u+Lu \in  E_i^{+}(y)$ implies
		$$\dist(u,E_i^{+}(y))\leq ||u-(u+L u)||\leq ||L||\cdot ||u||.$$
		Then, we have that $\delta(E_i^{+}(x),E_i^{+}(y))\leq ||L||$.
Next, given $v\in E_i^{+}(y)$, there exist $u\in  E_i^{+}(x)$ such that $v=u+Lu$. Since $u\in E_i^{+}(x), Lu\in E_i^{-}(x)$, it follows from (2) of Theorem \ref{T2} that $||u||\leq \ell||v||$. Thus, one has
		$$\dist(v,E_i^{+}(x))\leq ||v-u||\leq ||L||\cdot||u||\leq \ell||L||\cdot||v||.$$
		Therefore, $\delta(E_i^{+}(y),E_i^{+}(x))\leq \ell\cdot||L||$. This together with \eqref{eq:gap} yield that 		 $$\hat{d}(E_i^{+}(y),E_i^{+}(x))\leq 2\hat{\delta}(E_i^{+}(y),E_i^{+}(x)) \leq 2\ell||L||.$$
		
		To prove the other inequality in \eqref{bound-subsp},  note that for each $\beta> \delta(E_i^{+}(y),E_i^{+}(x))$ and for every $u\in E_i^{+}(x)\setminus\{0\}$, one has
		$$\dist(u+Lu, E_i^{+}(x))<\beta \cdot ||u+Lu||.$$
		Fix such a $u\in E_i^{+}(x)\setminus\{0\}$, there exists ${u}'\in E_i^{+}(x)$
		such that $||u+Lu-{u}'||< \beta \cdot ||u+Lu||$. Since $u-{u}'\in E_i^{+}(x)$, $Lu\in  E_i^{-}(x)$ and
 $K(x)\leq \ell$, by  (2) of Theorem \ref{T2}  we have that
		$$||Lu||< \ell ||u-{u}'+Lu||< \ell \beta \cdot ||u+Lu||< \ell \beta(1+||L|| )\cdot  ||u||.$$
		Thus, we conclude that
		$$\beta> \dfrac{||Lu||}{\ell(1+||L||)\cdot||u||}$$
		for  every $u \in E_i^{+}(x)\setminus\{0\}$. This implies that
		$$\beta\geq \frac{||L||}{\ell(1+||L||)}.$$
		By the arbitrariness of  $ \beta $ and \eqref{eq:gap}, one has
		$$\hat{d}(E_i^{+}(x),E_i^{+}(y))\geq \delta(E_i^{+}(x),E_i^{+}(y)) \geq \frac{||L||}{\ell(1+||L||)}. $$
This completes the proof of the lemma.
	\end{proof}

	Next, we shall prove Theorem \ref{main}.
	\begin{proof}[Proof of Theorem \ref{main}]
		By Lemma \ref{graph}, we have that
		$$|| L_{x,y} ||\rightarrow 0 \ \text{whenever} \ d(x,y)\rightarrow 0.$$
		Hence, there exists $\delta_2\in (0,\delta_1)$ with $\delta_2^{\nu}<1/4$ such that $||L_{x,y}||<1/2$ for any $x,y \in \Lambda_{\ell}$ with $d(x,y)<\delta_2$.
		
		Fix $x,y \in \Lambda_{\ell}$ with $d(x,y)<\delta_2$. Let $\Phi_{x,y}=Id+L_{x,y}$ be the isomorphism from $E_i^{+}(x)$ to $E_i^{+}(y)$. It is easy to see that $||\Phi_{x,y}||\leq 1+||L_{x,y}||$. Since $||L_{x,y}||<1/2$, we obtain that
		$\Phi_{x,y}^{-1}=Id + \sum_{k=1}^{\infty} (-L_{x,y})^{k}$. We write it as $\Phi_{x,y}^{-1}=Id + \hat{L}_{x,y}$, where $\hat{L}_{x,y}:E_i^{+}(y)\rightarrow E_i^{-}(x) $ and one can show that $||\hat{L}_{x,y}|| \leq  ||L_{x,y}||(1-||L_{x,y}||)^{-1}$.
		
		For simplicity of notations, denote  $L_{x,y}, \hat{L}_{x,y}$ and $\Phi_{x,y}^{\pm}$ by  $L, \hat{L}$ and $\Phi^{\pm}$ respectively. Note that
		$$E_i^{+}(x)=E_{i}(x) \oplus E_{i-1}^{+}(x)=\Phi^{-1} E_{i}(y) \oplus \Phi^{-1} E_{i-1}^{+}(y).$$
		By the triangle inequality we have
		\begin{equation}\label{eq:tri}
			\hat{d}(E_i(x),E_i(y)) \leq \hat{d}(E_i(x),\Phi^{-1} E_{i}(y))+\hat{d}(\Phi^{-1} E_{i}(y), E_i(y)).
		\end{equation}
		
		We first estimate $\hat{d}(\Phi^{-1} E_{i}(y), E_i(y))$ in \eqref{eq:tri}. For each $u\in  E_i(y)$ with $||u||=1$, we have that
		$$\dist(u,\Phi^{-1} E_{i}(y)) \leq ||u-\Phi^{-1}u||\leq ||\hat{L} u||\leq \dfrac{||L||}{1-||L||}\leq 2||L||.$$
		Similarly, for each $v \in  \Phi^{-1} E_{i}(y)$ with $||v||=1$ one has that
		$$\dist(v,E_{i}(y)) \leq ||v -\Phi v||\leq ||Lv||\leq ||L||.$$
		Consequently, we obtain that  $\hat{\delta}(\Phi^{-1} E_{i}(y), E_i(y))\leq2||L||$. By \eqref{eq:gap} and Lemma \ref{graph}, we have
		\begin{equation*}
			\hat{d}(\Phi^{-1} E_{i}(y), E_i(y))\leq 4||L|| \leq 4\ell(1+||L||)\hat{d}(E_i^{+}(x),E_i^{+}(y))\leq 6\ell\hat{d}(E_i^{+}(x),E_i^{+}(y))
		\end{equation*}
where the last inequality use the fact that $\|L\|\le 1$.
		Combining the above inequality and \eqref{eqa:EUnstable}, we have that
		\begin{equation}\label{eq:sec}
			\hat{d}(\Phi^{-1} E_{i}(y), E_i(y))\leq 6\ell
			C_{i}^{+} d(x,y)^{\nu_{i}^{+}}.
		\end{equation}
		
		Next we estimate $\hat{d}(E_i(x),\Phi^{-1} E_{i}(y))$.  Let $A_n=\cA(x,n), B_n=\cA(y,n) \circ \Phi,E=E_{i}(x),F=\Phi^{-1} E_{i}(y)$ and ${E}'=E_{i-1}^{+}(x), {F}'=\Phi^{-1} E_{i-1}^{+}(y)$.
		
		For each $u\in F=\Phi^{-1}E_{i}(y)$ and each $n\in \N$, by (1) of Theorem \ref{T2}
		\begin{align}\label{B1}
\begin{split}
			|| B_n u ||=||\cA(y,n)  \Phi u||&\leq \ell e^{(\lambda_{i}+\varepsilon)n} \cdot || \Phi u ||\\
			&\leq ||\Phi||  \ell e^{(\lambda_{i}+\varepsilon)n} \cdot || u ||\\
			&\leq (1+||L||) \ell e^{(\lambda_{i}+\varepsilon)n} \cdot || u || \\
			&\leq 2 \ell e^{(\lambda_{i}+\varepsilon)n} \cdot || u ||
\end{split}
		\end{align}
		where we use the fact that $||L||<1/2$ in the last inequality.
		
		For each $u\in {F}'=\Phi^{-1} E_{i-1}^{+}(y)$ and each $n\in \N$, by (1) of  Theorem \ref{T2} we have
		\begin{align}\label{B2}
\begin{split}
			|| B_n u ||&=||\cA(y,n)  \Phi u||\\
&\geq  \ell^{-1} e^{(\lambda_{i-1}-\varepsilon)n} \cdot || \Phi u ||\\
			&\geq ||\Phi^{-1}||^{-1}  \ell^{-1} e^{(\lambda_{i-1}-\varepsilon)n} \cdot || u ||\\
			&\geq (1-||L||)  \ell^{-1} e^{(\lambda_{i-1}-\varepsilon)n} \cdot || u || \\
			&\geq (2 \ell)^{-1} e^{(\lambda_{i-1}-\varepsilon)n} \cdot || u ||
\end{split}
		\end{align}
where the third inequality uses the fact that  $||\Phi^{-1}||\leq 1+||\hat{L}|| \leq (1-||L||)^{-1}$.
		
		For every $u\in F, v\in {F}'$, then $ \Phi u \in E_{i}(y),  \Phi v \in E_{i-1}^{+}(y)$ and, by (2) of Theorem \ref{T2} one has that
		$$\ell|| \Phi u + \Phi v ||\geq \max\{ || \Phi u ||,|| \Phi v ||\}.$$
		This yields that
		$$\ell||\Phi|| \cdot ||u+v|| \geq ||\Phi^{-1}||^{-1} \cdot \max\{||u||,||v||\}.$$
		Since $ ||L||<1/2 $, we have that $||\Phi||\cdot ||\Phi^{-1}||\leq (1+||L||)(1-||L||)^{-1} \leq 3$. Therefore, one has that
		\begin{eqnarray} \label{B3}
\max\{||u||,||v||\}\le 3\ell ||u+v||.
\end{eqnarray}
	Consider $\alpha_{1}=e^{\lambda_{i-1}-\varepsilon},\alpha_{2}=e^{\lambda_{i}+\varepsilon}$ and replace $\ell$ by $3\ell$.	 Clearly, the sequence $\{A_n\}$ satisfies condition (i) of Lemma \ref{L5}, and it follows from \eqref{B1}, \eqref{B2} and \eqref{B3}  that the condition  (ii) and (iii) of Lemma \ref{L5} hold for every $n\in \N$.

		At last, consider a constant $a>\sup_{x\in X}||A(x)||$ as in  Lemma \ref{L4}, one can show that
		\begin{align*}
			||A_n-B_n||&\leq ||\cA(x,n)- \cA(x,n) \Phi||+||(\cA(x,n)-\cA(y,n))\Phi||\\
			&\leq ||\cA(x,n)||\cdot ||L||+||\cA(x,n)-\cA(y,n) ||\cdot (1+||L||)\\
			&\leq a^n ||L||+ 2a^n d(x,y)^\nu\\
			&\leq a^n (||L||+2d(x,y)^\nu).
		\end{align*}
		Let $\delta=||L||+2d(x,y)^\nu<1$, since $a>\alpha_{1}>\alpha_2$ there exist $\hat{n}\in \N$ such that
		$$(\frac{\alpha_2}{a})^{\hat{n}+1}\le \delta < (\frac{\alpha_2}{a})^{\hat{n}}.$$
		It follows from Lemma \ref{L5} that
		\begin{equation}\label{pre-h}
			\hat{d}(E_i(x),\Phi^{-1} E_{i}(y)) \leq \hat{C}_i (||L||+2d(x,y)^\nu)^{\hat{\nu}_{i}}
		\end{equation}
		where $\hat{C}_i=(4+6\ell)(3\ell)^2 e^{\lambda_{i-1}-\lambda_{i}-2\varepsilon}$ and $\hat{\nu}_{i}= (\lambda_{i-1}-\lambda_{i}-2\varepsilon)/(\log a-\lambda_{i}-\varepsilon )<1$. By Lemma \ref{graph} and \eqref{eqa:EUnstable}, there exist  constants $C_{i}^{+}>0$ and $\nu_{i}^{+}\in(0,\nu )$ so that
		$$||L||\leq \ell(1+||L||)\hat{d}(E_i^{+}(x),E_i^{+}(y))\leq 2\ell C_{i}^{+} d(x,y)^{\nu_{i}^{+}}.$$
Since $\nu_{i}^{+}<\nu $ implies $d(x,y)^{\nu_{i}^{+}}\geq d(x,y)^{\nu}$, this together with \eqref{pre-h} one has
		\begin{equation}\label{eq:fst}
			\hat{d}(E_i(x),\Phi^{-1} E_{i}(y)) \leq 2\hat{C}_i(\ell C_{i}^{+}+1)d(x,y)^{\nu_{i}^{+} \hat{\nu}_{i}}.
		\end{equation}
		Combining \eqref{eq:tri}, \eqref{eq:sec} and \eqref{eq:fst},  we have that
		\begin{equation}
			\begin{aligned}
				\hat{d}(E_i(x),E_i(y))&\leq 6\ell C_{i}^{+} d(x,y)^{\nu_{i}^{+}}+ 2\hat{C}_i(\ell C_{i}^{+}+1)d(x,y)^{\nu_{i}^{+}\hat{\nu}_{i}}\\
				&\leq C_{i} d(x,y)^{\nu_{i}}
			\end{aligned}
		\end{equation}
where $\nu_{i}=\nu_{i}^{+}\hat{\nu}_{i}$ and $C_i=6\ell C_{i}^{+}+ 2\hat{C}_i(\ell C_{i}^{+}+1)$.
		This completes the proof.
	\end{proof}

   \section{The  non-invertible case}\label{non-inert}
    In this section, assume that $f$ is a non-invertible Lipschitz map on a compact metric space $(X,d)$, and $\mu$ is an $f$-invariant ergodic Borel probability measure on $X$.

    Let $\cB$ be a separable Banach space, and let $L(\cB)$ denote the space of all bounded linear operators on $\cB$. Let $A:X\rightarrow L(\cB)$ be a $\nu$-H\"{o}lder continuous operator valued function, i.e., there exist constants $C$ and $\nu$ such that $||A(x)-A(y)||\leq Cd(x,y)^{\nu}$. In addition, suppose that $\alpha(A,\mu)<\lambda(A,\mu)$, i.e.,  $A$ is quasi-compact (see the definitions in Section \ref{LC}).

    The one-sided Oseledets theorem was obtained by Doan \cite[Theorem 7.1.7]{Doan09} that we will recall in below, see \cite{Lian10} for the  two-sided case.
    \begin{theorem} \label{One MET}
    Let $f$ be a continuous map on a compact metric space $(X,d)$, and let $\mu$ be an $f$-invariant ergodic Borel probability measure. Given a linear cocycle $\cA$ over $f$ generated by a quasi-compact and strongly measurable operator valued function $A:X \rightarrow L(\cB)$. Then there exists a $f$-invariant subset $X_0\subset X$ of  full $\mu$-measure such that there exist $k$ ($k\in \N \cup \infty$) numbers $\lambda_{1}>\cdots>\lambda_{k}>\alpha(A,\mu)$, for every $x\in X_0$ there is a filtration
    $$\cB=V_1(x)\supset V_2(x) \supset \cdots \supset V_{k+1}(x)$$
    with the following properties:
    \begin{enumerate}
        \item[(a)] if $k<\infty$, let $\lambda_{k+1}=\alpha(A,\mu)$, for each $i\in \{1,\cdots,k\}$, $V_{i}(x)$ is a closed, finite-codimensional subspace and co-$\dim V_i(x)=m_i$ is a finite constant. Moreover, $A(x)V_{i}(x)\subset V_{i}(f(x))$, and for every $u\in V_{i}(x)\setminus V_{i+1}(x)$ we have that
    		$$ 
    		\lim_{n\rightarrow \infty} \frac{1}{n}\log ||\cA(x,n)u||=\lambda_{i}\ \text{and}\   	 \limsup_{n\rightarrow \infty} \frac{1}{n}\log ||\cA(x,n)|_{V_{i+1}(x)}||\le \lambda_{i+1}; $$
    	\item[(b)] if $k=\infty$ then $\displaystyle{\lim_{i\rightarrow \infty}\lambda_{i}=\alpha(A,\mu)}$, and $V_{i}(x)$  is a closed, finite-codimensional and co-$\dim V_i(x)=m_i$ is a finite constant for each $i\in \N$. Moreover, $A(x)V_{i}(x)\subset V_{i}(f(x))$, and for every $u\in V_{i}(x)\setminus V_{i+1}(x)$ we have that
    	$$ 
    	\lim_{n\rightarrow \infty} \frac{1}{n}\log ||\cA(x,n)u||=\lambda_{i} \  \text{and}\  	 \limsup_{n\rightarrow \infty} \frac{1}{n}\log ||\cA(x,n)|_{V_{i+1}(x)}||\le \lambda_{i+1}.$$
    \end{enumerate}
    \end{theorem}
   \begin{remark}
   	Notice that the continuity of $A:X\to L(\mathcal{B})$  implies that $A$ is strongly measurable, and the map $x\mapsto V_{i}(x)$ is measurable (see \cite[Remark 2.9]{Quas12}). 
   \end{remark}
   
 
  In the following, we will show that the map $x\mapsto V_{i}(x)$ is (locally) H\"{o}lder continuous on a compact set of arbitrarily large measure provided that $f$ is Lipschitz and $A:X\rightarrow L(\cB)$ is H\"{o}der continuous.

  \begin{theorem}\label{main2}
  	Let $f$ be a Lipschitz map on a compact metric space $(X,d)$, and let $\mu$ be an $f$-invariant ergodic Borel probability measure. Given a linear cocycle $\cA$ over $f$ generated by a quasi-compact and $\nu$-H\"{o}der continuous operator valued function $A:X \rightarrow L(\cB)$, where $(\cB,||\cdot||)$ is a separable Banach space. Then, for every $\gamma>0$ and every subspace $V_{i}(x)$ as in Theorem \ref{One MET}, there exists a compact subset $\Lambda_\gamma\subset X$ with $\mu(\Lambda_\gamma)>1-\gamma$ so that the map $x\mapsto V_{i}(x)$ is (locally) H\"{o}lder continuous on $\Lambda_\gamma$.
  \end{theorem}
  
  If $k$ is a finite positive integer, we will show that the previous theorem for every $i=1,2,\cdots,k+1$. If $k$ is infinite, we will prove the same result for every $i\in \mathbb{N}$.
  In the following, we only prove  the  first case,  the latter  can be proven in a similar fashion.

  The following lemma proved in \cite[Lemma 2.11]{Quas12}  allows us to choose a ``good" complementary space for every space $V_{i}(x)$. In the setting of the space $\R^{d}$ or the Hilbert space, one may usually choose the orthogonal complement $V_{i}(x)^{\perp}$.
  
  \begin{lemma} \label{L41}
  	Let the filtration $\mathcal{B}=V_1(x)\supset \cdots \supset V_{k+1}(x)$  be as in Theorem \ref{One MET}. Then, for every $1\leq i \leq k$, there exists a finite-dimensional subspaces $\widetilde{U_i}(x)$ such that
  	\begin{enumerate}
  		\item[(1)] for $\mu$-almost every $x$, $ V_{i+1}(x)\oplus \widetilde{U_i}(x)=V_{i}(x)$ and the map $x\mapsto  \widetilde{U_i}(x)$ is measurable;
  		\item[(2)] let $U_i(x)=\bigoplus_{k=0}^{i-1}\widetilde{U_k}(x)$ (where $ \widetilde{U_0}(x)=\{0\} $), and let $\pi_{i}^{u}(x)$ and $\pi_{i}^{v}(x)$ denote the projections of $\cB$ onto $U_i(x)$ and $V_{i}(x)$ via the splitting $V_{i}(x)\oplus U_i(x)$ respectively, then
  		$$||\pi_{i}^{v}(\cdot)||,||\pi_{i}^{u}(\cdot)||\in L^{\infty}(\mu).$$
    \end{enumerate}	
  \end{lemma}
 Since $V_1(x)=\mathcal{B}$,  the statement in Theorem \ref{main2} clearly holds.  In the following, fix $i\in \{2\cdots,k+1\}$.
  By (2) of Lemma \ref{L41}, on can choose $\ell>0$, such that 
  $$ \max \{||\pi_{i}^{u}(x)||,||\pi_{i}^{v}(x)||\}\leq \ell$$
  for $\mu$-almost every $x\in X$.
  Reducing $X_0$ by a zero measure set such that $X_0$ remains $f$-invariant and for every $x\in X_0$ the above inequality holds for every $x\in X_0$. Then, for every $x\in X_0$, one has
  \begin{equation}\label{eq:Onenorm}
  \max\{||u||,||v||\}\leq \ell ||u+v||
  \end{equation}
  for every $u\in U_{i}(x)$ and every $v\in V_i(x)$. By Theorem \ref{One MET},  for every $x\in X_0$ and every $u\in U_{i}(x)\setminus \{0\}$ we have that 
  \begin{equation}\label{eq:ui}
    \lim_{n\rightarrow \infty}\frac{1}{n} \log||\cA(x,n) u||\geq \lambda_{i-1}.
  \end{equation}

  Fix $x\in X_0$, let $B(x)=\pi^{u}_{i}(f(x))\circ A(x)\circ\pi^{u}_{i}(x)$, $C(x)=\pi^{v}_{i}(f(x))\circ A(x)\circ \pi^{u}_{i}(x)$ and $D(x)=\pi^{v}_{i}(f(x))\circ A(x)\circ \pi^{v}_{i}(x)$. Since $\pi^{u}_{i}(f(x))\circ A(x)\circ \pi^{v}_{i}(x)=0$ by the  invariance of $V_i(x)$,  we have that
  $$A(x)=B(x)+C(x)+D(x).$$
  Similarly, define $B_{n}(x)$, $C_{n}(x)$ and $D_{n}(x)$ as  above, such that $\cA(x,n)=B_n(x)+C_n(x)+D_n(x)$. As in \cite[Lemma 2.12]{Quas12} ( see also  \cite[Sect. 4.2.5]{Viana 14}),  One can show that
  \begin{align*}
  B_n(x)=B(f^{n-1}(x))\circ \cdots \circ B(x),\,
  D_n(x)=D(f^{n-1}(x))\circ \cdots \circ D(x)
  \end{align*}  and
  $$C_n(x)=\sum_{j=0}^{n-1} D_{n-j-1}(f^{j+1}(x)) C(f^{j}(x)) B_j(x).$$
  Note that $B(x)|_{U_i(x)}:U_{i}(x)\rightarrow U_{i}(f(x))$, $D(x)|_{V_i(x)}=A(x)|_{V_i(x)}$ and $C(x)|_{U_i(x)}:U_{i}(x)\rightarrow V_{i}(f(x))$. For every $u\in U_{i}(x)$, by (\ref{eq:Onenorm}) we have that
  $$\max \{||B(x)u||,||C(x)u||\}\leq \ell ||B(x)u+C(x)u||=\ell ||A(x)u||.$$
  This implies that $$\max\{||C(x)||,||B(x)||\}\leq \ell||A(x)||.$$
   Moreover, if there exists  $u\in U_{i}(x)\setminus \{0\}$ such that $B(x)u=0$, then 
  $$\cA(x,n)u=\cA(fx,n-1)A(x)u=\cA(fx,n-1)C(x)u$$  for every $n>1$.
  Since $C(x)u\in V_{i}(f(x))$, one has 
  $$\limsup_{n\rightarrow \infty}\frac{1}{n}\log||\cA(x,n)u|| \le \lambda_{i}.$$ 
  This yields a contradiction with (\ref{eq:ui}). Hence,  one has $B(x)u\neq 0$ for each $u\in U_{i}(x)\setminus \{0\}$. Therefore, $B(x)$ is an isomorphism from $U_i(x)$ to $U_i(fx)$.  Note that $\dim U_i(x)=\dim U_i(fx)<\infty$, restricted on $\{U_i(x)\}$,  $\{B_n(x)\}$ is a cocycle on $X$ with respect to $f$,  applying the multiplicative ergodic theorem of the finite dimensional case (e.g., see \cite{Viana 14}), one can obtain that for $\mu$-almost every $x\in X$
  the  following limit 
  \begin{eqnarray*}
  \lim_{n\rightarrow \infty} \frac{1}{n}\log||B_n(x)u||
  \end{eqnarray*}
   exist for every $u\in U_i(x)\setminus\{0\}$. Moreover, by \cite[Sub-lemma 2.13]{Quas12} or \cite[Proposition 4.14]{Viana 14}, we have that for every $u\in U_i(x)\setminus\{0\}$
  $$\lim_{n\rightarrow \infty} \frac{1}{n}\log||B_n(x)u||=\lim_{n\rightarrow \infty} \frac{1}{n}\log||\cA(x,n)u|| \geq \lambda_{i-1}.$$
  This implies that 
  \begin{equation}\label{Blim}
  \lim_{n\rightarrow \infty} \frac{1}{n}\log||(B_n(x)|_{U_{i}(x)})^{-1}||^{-1}\geq \lambda_{i-1},\,\, \mu\text{-almost every } x.
  \end{equation}
 Without loss of generality, assume that \eqref{Blim} holds for every $x\in X_0$.
\begin{lemma}\label{L42}
	Fix $x\in X_0$ and a small number $\varepsilon>0$, there exists  $m>0$ such that for each $v\in U_{i}(x)\setminus\{0\}$ with $||v||=1$, we have $||\cA(x,n) v||\geq e^{n(\lambda_{i-1}-\varepsilon)}||v||$ for each $n\geq m$.
\end{lemma}
\begin{proof}
    For every $u\in U_{i}(x)\setminus\{0\}$, write $\cA(x,n)u=B_n(x)u+C_n(x)u$. Note that $B_n(x)u\in U_i(f^{n}x)$ and $C_n(x)u\in V_i(f^{n}x)$, it follows from \eqref{eq:Onenorm} that
    $$||\cA(x,n)u||\geq \ell^{-1} ||B_n(x)u||\geq \ell^{-1}||(B_n(x)|_{U_{i}(x)})^{-1}||^{-1}||u||.$$
    Therefore, we have that
    $$\liminf_{n\rightarrow \infty}\frac{1}{n}\inf_{\substack{u\in U_{i}(x)\\ ||u||=1}}\log||\cA(x,n)u||\geq \lim_{n\rightarrow \infty} \frac{1}{n}\log||(B_n(x)|_{U_{i}(x)})^{-1}||^{-1}\geq \lambda_{i-1}.$$
    This implies the desired result immediately. 
\end{proof}

 Next, we will show that the map $x\mapsto V_{i}(x)$ is (locally) H\"{o}lder continuous on a compact subset of arbitrarily large measure.

  \begin{proof}[Proof of Theorem \ref{main2}] Fix $i\in\{2,\cdots,k+1\}$ and a sufficiently small number $\varepsilon>0$. For each $n\in \mathbb{N}$, let
  \begin{align*}
  A^{i}_{n,\varepsilon}=\{x\in X_0: ||\cA(x,m) u||\leq e^{m(\lambda_{i}+\varepsilon)}||u||,  \quad \forall u \in  V_{i}(x), \ \forall m\geq n\}
  \end{align*}
  and 
  \begin{align*}
  B^{i}_{n,\varepsilon}=\{x\in X_0: ||\cA(x,m) v||\geq e^{m(\lambda_{i-1}-\varepsilon)}||v||,  \quad  \forall v \in  U_{i}(x),\  \forall m\geq n\}.
  \end{align*}
  Clearly, the sequences of sets $\{A^{i}_{n,\varepsilon}\}$ and $\{B^{i}_{n,\varepsilon}\}$ are nested, and by Theorem \ref{One MET} and Lemma \ref{L42} we have that
  $$\mu\Big(\bigcup_{n=1}^{\infty} A^{i}_{n,\varepsilon}\Big)=\mu\Big(\bigcup_{n=0}^{\infty} B^{i}_{n,\varepsilon}\Big)=1.$$ 
  Therefore, for every $\gamma>0$ there exists $n_0$ such that $\mu (A^{i}_{n,\varepsilon}\cap  B^{i}_{n,\varepsilon})>1-\gamma$ for every $n\geq n_0$. 

  Let $\Lambda_{\gamma}=A^{i}_{n_0,\varepsilon}\cap  B^{i}_{n_0,\varepsilon}$. We may assume further that $\Lambda_{\gamma}$ is compact since otherwise we can
approximate it from within by a compact subset. For every $x,y\in \Lambda_{\gamma}$ with $d(x,y)^{\nu}<(\frac{e^{\lambda_{i}+\varepsilon}}{a})^{n_0}<1$, where  $a$ is a sufficiently large constant  as in Lemma \ref{L4}. Note that Lemma \ref{L4} is also valid for $n>0$ in this case. Applying Lemma \ref{L5} with $A_n=\cA(x,n), B_n=\cA(y,n), \alpha_{2}=e^{\lambda_{i}+\varepsilon},\alpha_{1}=e^{\lambda_{i-1}-\varepsilon}, E=V_i(x),F=V_i(y),{E}'=U_i(x)$ and ${F}'=U_i(y)$. By the construction of the set $\Lambda_{\gamma}$, for every $n\geq n_0$, the conditions (i), (ii) and (iii) of Lemma \ref{L5} hold. Set $\delta=d(x,y)^{\nu}$,  there exists a number ${n}'\geq n_0$ such that
  $$(\frac {e^{\lambda_{i}+\varepsilon}} {a})^{{n}'+1}\le \delta < (\frac {e^{\lambda_{i}+\varepsilon}}{a})^{{n}'}.$$
  It follows from Lemma \ref{L5} that
  \begin{equation*}
  \hat{d}(V_i(x),V_i(y))\leq C_{i} d(x,y)^{\nu_{i}}
  \end{equation*}
  where $C_{i}=(4+2\ell)\ell^2e^{\lambda_{i-1}-\lambda_{i}-2\varepsilon}$ and $\nu_{i}^{-}=\nu(\lambda_{i-1}-\lambda_{i}-2\varepsilon)/(\log a-\lambda_{i}-\varepsilon )<\nu<1.$
  This completes the proof.
  \end{proof}
	


	

\begin{thebibliography}{2}
		\bibitem{Simion16}
		V. Araujo, A. I. Bufetov and S. Filip, On H\"{o}lder-continuity of Oseledets subspaces, J. Lond. Math. Soc. 2016, 93: 194--218.
		
		\bibitem{Pesin07}
		L. Barreira, Y. Pesin, Nonuniform Hyperbolicity (Encyclopedia of Mathematics and its Applications, 115), Cambridge University Press, Cambridge, 2007.
		
		
		
		\bibitem{Young17}
		A. Blumenthal, L.-S. Young, Entropy, volume growth and SRB measures for Banach space
		mappings, Invent. Math., 2017, 207: 833--893.
		
		\bibitem{Brin01}
		M. Brin, H\"{o}lder Continuity of Invariant Distributions (Smooth Ergodic Theory and its Applications), American Mathematical Society, Providence, RI, 2001, pp. 91--93.
		
		\bibitem{Doan09}
		T. S. Doan, Lyapunov exponents for random dynamical systems, PhD Thesis,
		Fakult\"{a}t Mathematik und Naturwissenschaften der Technischen Universit\"{a}t Dresden, 2009.
		
		\bibitem{Froyland18}
		D. Dragi\v{c}evi\'{c}, G. Froyland, H\"{o}lder continuity of Oseledets splittings for semi-invertible operator cocycles, Ergodic Theory and Dynamical Systems, 2018, 38(03): 961--981.
		
		\bibitem{Froyland10}
		G. Froyland, S. Lloyd and A. Quas, Coherent structures and isolated spectrum for Perron-Frobenius cocycles, Ergodic Theory and Dynamical Systems, 2010, 30(3): 729--756.
		
		\bibitem{Froyland13}
		G. Froyland, S. Lloyd, A, Quas, A semi-invertible Oseledets theorem with applications to transfer operator cocycles, Discrete and Continuous Dynamical Systems, 2013, 33(9): 3835--3860.
		
		\bibitem{FK60}
		H. Furstenberg, H, Kesten, Products of random matrices, Ann. Math. Statist., 1960, 31: 457--469.
		
		\bibitem{Quas12}
		C. Gonz\'{a}lez-Tokman, Anthony Quas, A semi-invertible operator Oseledets theorem, Ergodic Theory and Dynamical Systems, 2014, 34(4): 1230--1272.
		
		\bibitem{Kato95}
		T. Kato, Perturbation Theory for Linear Operators, Springer, Berlin, 1995.
		
		\bibitem{Lian10}
		Z. Lian, K. Lu, Lyapunov exponents and invariant manifolds for random dynamical systems in a banach space, Mem. Amer. Math. Soc., 206, (2010), no. 967.
		
		\bibitem{Lucas17}
		B. Lucas, D. Davor, Periodic approximation of exceptional Lyapunov exponents for semi-invertible operator cocycles, Ann. Acad. Sci. Fenn. Math., 2019, 44(1): 183--209.
		
		\bibitem{M83}
		R. Man\'{e}, Lyapounov exponents and stable manifolds for compact transformations. Geometric Dynamics
		(Rio de Janeiro, 1981) (Lecture Notes in Mathematics, 1007). Springer, Berlin, 1983, pp. 522--577.
		
		\bibitem{Oseledec68}
		V. I. Oseledec, A multiplicative ergodic theorem. Characteristic Lyapunov, exponents of dynamical systems, Trudy Moskov. Mat. Obsc., 1968, 19: 179--210,
		
		\bibitem{Ruelle82}
		D. Ruelle, Characteristic exponents and invariant manifolds in Hilbert space, Annals of Math., 1982, 115:243--290.
		
		\bibitem{S90}
		K.-U. Schauml\"{o}ffel, Zuf\"{a}llige Evolutionsoperatoren f\"{u}r stochastische partielle Differentialgleichungen,
		Dissertation, Universit\"{a}t Bremen, 1990.
		
		\bibitem{SF91}
		K.-U. Schauml\"{o}ffel, F. Flandoli, A multiplicative ergodic theorem with applications to a first order
		stochastic hyperbolic equation in a bounded domain, Stoch. Stoch. Rep., 1991, 34(3-4): 241--255.
		
		\bibitem{Thieullen87}
		P. Thieullen, Fibres dynamiques asymptotiquement compacts exposants de Lyapounov. Entropie. Dimension, Ann. Inst. H. Poincar\'{e} Anal. Non Lin\'{e}aire., 1987, 4(1): 49--97.
		
		
		
		
		\bibitem{Viana 14}
		M. Viana, Lectures on Lyapunov exponents (Vol. 145 of Cambridge studies in advanced mathematics), Cambridge University Press, Cambridge, 2014.
		
		
		
		
		
		
		
		
		
		
	\end{thebibliography}

\end{document}